\documentclass[12pt]{amsart}

\numberwithin{equation}{section}

\newtheorem{theorem}[subsection]{Theorem}
\newtheorem{stheorem}[subsubsection]{Theorem}
\newtheorem{sremark}[subsubsection]{Remark}
\newtheorem{slemma}[subsubsection]{Lemma}
\newtheorem{lemma}[subsection]{Lemma}
\newtheorem{proposition}[subsection]{Proposition}
\newtheorem{sproposition}[subsubsection]{Proposition}
\newtheorem{corollary}[subsection]{Corollary}

\newtheorem{remark}[subsection]{Remark}

\theoremstyle{definition}
\newtheorem*{ack}{Acknowledgements}

\theoremstyle{remark}

\newcommand{\ztwo}{{\bf{Z}}/2}

\newcommand{\field}{{\bf{F}}}

\newcommand{\tr}{\mathop{\rm Tr}}

\newcommand{\lt}{\mathop{\rm LT}}
\newcommand{\lm}{\mathop{\rm LM}}

\newcommand{\ring}{S}

\newcommand{\del}{\Delta}

\newcommand{\tat}{t\^ete-a-t\^ete}

\newcommand{\shsop}[1]{{N_{#1}}}
\newcommand{\norm}[1]{N_G(#1)}

\title[Klein four]
{Rings of invariants for modular representations of the Klein four group}

\author{M\"ufit Sezer}
\address{Department of Mathematics, Bilkent University,
06800 Ankara, Turkey}
\email{sezer@fen.bilkent.edu.tr}

\author{R.~James Shank}
\address{School of Mathematics, Statistics and Actuarial Science \\
 \hfil\break\indent University of Kent, Canterbury, CT2 7NF, United Kingdom}
\email{R.J.Shank@kent.ac.uk}

\thanks{The first author is  partially supported by a grant from T\"UBITAK: 112T113}
\date{\today}

\begin{document}
%\begin{abstract}
%We investigate the relationship \ldots
%\end{abstract}

\maketitle
\begin{abstract}
We study the rings of invariants for the indecomposable
modular representations of the Klein four group.
For each such representation we compute the Noether number
and give minimal generating sets for the Hilbert ideal and the field of fractions.
We observe that, with the exception of the regular representation,
the Hilbert ideal for each of these representations is a complete intersection.
\end{abstract}

\section*{Introduction}
The modular representation theory of the Klein four group has long attracted attention.
The group algebra of Klein four over an infinite field of characteristic $2$ is one of
the relatively rare examples of a group algebra with domestic representation type
(see, for example, \cite[\S 4.4]{benson}).
If we work over an algebraically closed field then for each even dimension there
is a one parameter family of indecomposable representations and a finite number of exceptional
indecomposable representations. For each odd dimension (greater than $1$) there are only two
indecomposable representations. In this paper we investigate the rings of invariants of the
indecomposable representations of the Klein four group over fields of characteristic $2$.
For each such representation we compute the Noether number and give minimal generating sets
for the Hilbert ideal and the field of fractions (definitions are given below).
For an indecomposable representation of the Klein four group, say $V$, our results show that
the Noether number is at most $2\dim(V)+1$
(detailed formulae are given later in this introduction) and, with the exception of the regular representation,
the Hilbert ideal is generated by a homogeneous system of parameters.
We note that the Hilbert ideals are generated by polynomials of degree at most $4$,
confirming Conjecture~3.8.6(b) of \cite{DK} for these representations.

We start with a few definitions and some notation.
Suppose that $V$ is a finite dimensional representation of a finite group
$G$ over a field $\field$. The induced action on the dual space
$V^*$ extends to the symmetric algebra $S(V^*)$ of polynomial
functions on $V$ which we denote by $\field[V]$. The action of
$g\in G$ on $f\in \field[V]$ is given by $(gf)(v)=f(g^{-1}v)$ for
$v\in V$. The ring of invariant polynomials $$\field[V]^G=\{f\in
\field[V] \mid \; g(f)=f \; \forall g \in G\}$$ is a graded,
finitely generated subalgebra of $\field[V]$.
The maximal degree of a polynomial in a
minimal homogeneous generating set for $\field[V]^G$ is known as the {\it Noether
number} of $V$. The ideal in $\field[V]$ generated by the homogeneous invariants of
positive degree is the {\it Hilbert ideal} of $V$.
If the characteristic of $\field$
divides $|G|$, then $V$ is called a {\it modular} representation.
Rings of invariants for non-modular representations are reasonably well-behaved.
For instance, it is well-known  that if $V$ is non-modular, then $\field[V]^G$ is always
Cohen-Macaulay and the Noether number is less than or equal to $|G|$
(see, for example, \cite[\S 3.4, \S 3.8]{DK}).
Both of these properties can fail in the modular case. Rings of invariants for
modular representations are rarely Cohen-Macaulay and there is no bound on
the degrees of a generating set which depends only on the group order.
Computing rings of invariants for modular representations can be difficult even in basic cases.
Consider a representation of a
cyclic $p$-group ${\bf{Z}}/ {p^r}$ over a field of characteristic
$p$. The action is easy to describe: up to a change of basis, a
generator of the group acts by a sum of Jordan blocks each with eigenvalue $1$ and size at
most $p^r$. Despite this, even when $r=1$, although the Noether numbers are known \cite{FSSW},
an explicit generating set has been constructed for only a limited
number of cases; see \cite{wehlau} for a summary and
recent advances. For $r>1$, much less is known; see \cite{sw-prog} for the study of a specific case and
\cite{sezer} for some partial results on degree bounds. This paper is a part of a programme, initiated in
 \cite{CSW}, to understand the rings of invariants of modular representations of elementary abelian $p$-groups.
In \cite{CSW}, the rings of invariants of all two dimensional
representations and all three dimensional representations for groups of rank at
most three were computed; in all cases the rings were shown to be complete intersections.

The results in section \ref{block} apply to an arbitrary group
$G$ but for the rest of the paper $G:=\langle\sigma_1,\sigma_2\rangle\cong\ztwo \times \ztwo$ denotes the Klein four
group.
For $\field$ an algebraically closed field of characteristic $2$,
the indecomposable representations of the Klein four group over $\field$ are the following:
\begin{itemize}
\item the trivial representation $\field$;
\item the regular representation $V_{reg}$;
\item a representation of dimension $2m$ for each $\lambda\in\field\cup\{\infty\}$, which we denote by $V_{m,\lambda}$;
\item the representations $\Omega^m(\field)$ and $\Omega^{-m}(\field)$ of dimension $2m+1$, where $\Omega$ denotes the Heller
operator.
\end{itemize}
See \cite[\S 4.4]{benson} for a detailed discussion of this classification.
Note that $V_{m,0}$, $V_{m,1}$ and $V_{m,\infty}$, while not equivalent representations, are linked by group automorphisms.
Therefore the invariants can be computed using the same matrix group and
$\field[V_{m,0}]^G\cong\field[V_{m,1}]^G\cong\field[V_{m,\infty}]^G$.
In  \cite{El-Fl}, the depth of  $\field[V]^G$ was computed for
each of the indecomposable modular representations of the Klein four group.
The only indecomposable representations for which the ring of invariants is Cohen-Macaulay
are the the trivial representation, the regular representation,
$V_{1,\lambda}$, $V_{2,\lambda}$, $\Omega^{-1}(\field)$, $\Omega^{-2}(\field)$ and $\Omega^{1}(\field)$.
Note that, for each of these representations, $\field[V]^G$ is a complete intersection.
In \cite{K-S} separating sets of invariants are given for the indecomposable modular representations of the Klein four group.

We identify $\field[V]$ with the polynomial algebra  on the
variables $x_i$ and $y_j$. We use the graded reverse lexicographic order (grevlex)
with $x_i<y_j$, $x_i<x_{i+1}$ and $y_j<y_{j+1}$.
We adopt the convention that a monomial is a product of variables and a term is a monomial multiplied by a coefficient.
For a polynomial $f\in \field[V]$, we denote the leading monomial
by $\lm(f)$ and the leading term by $\lt(f)$. We make occasional use of SAGBI bases,
the  {\bf S}ubalgebra {\bf A}nalog of
a {\bf G}r\"obner {\bf B}asis for {\bf I}deals.
For a subset $\mathcal B=\{h_1\dots,h_{\ell}\}$ of a subalgebra $A\subset\field[V]$
and a sequence $I=(i_1,\ldots,i_{\ell})$ of non-negative integers,
denote $\prod_{j=1}^{\ell}h_j^{i_j}$ by $h^I$.
A {\it \tat}\ for $\mathcal B$ is a pair $(h^I,h^J)$ with $\lm(h^I)=\lm(h^J)$;
 we say that a \tat\ is non-trivial if the support of $I$ is disjoint from the support of $J$.
The reduction of an S-polynomial is a fundamental calculation in the theory of Gr\"obner bases.
The analogous calculation for SAGBI bases is the {\it subduction} of a \tat.
$\mathcal B$ is a SAGBI basis for $A$ if every non-trivial \tat\ subducts to zero.
A SAGBI basis is a particularly useful generating set for the subalgebra.
For background material on SAGBI bases,
see \cite[\S 11]{sturm} or \cite[\S 3]{sw-comp}.
 For $f\in \field[V]$, we define
the {\it transfer} of $f$ by
$\tr(f):= \sum_{\sigma\in G}\sigma (f)$ and the {\it norm} of $f$, which we denote
by $N_G(f)$, to be the product over the $G$-orbit of $f$.
If the coefficient of a monomial $M$ in a polynomial $f$ is non-zero, we say that
$M$ {\it appears} in $f$.

We conclude the introduction with a summary of the paper.
Section~\ref{preliminaries} contains preliminary results on the invariant theory of $\ztwo$.
In section~\ref{block}, we introduce the concept of a {\it block hsop}, a particularly nice homogeneous system of parameters,
and prove a theorem which we use to compute Noether numbers. A recent result of Peter Symonds
\cite[Corollary 0.3]{symonds} is a key ingredient in our proof. The results of this section are valid for
any modular representation of a finite group.

In section~\ref{even-sec}, we consider the even dimensional representations. We include an explicit description of the group actions.
We show that for $m>1$, the Noether number of $V_{m,\lambda}$ is $3m-2\lfloor m/2\rfloor$
if $\lambda\in\field\setminus\field_2$ and $3m-2\lceil m/2\rceil$ if $\lambda\in\{0,1,\infty\}$.
We also show that the Hilbert ideal of $V_{m,\lambda}$ is generated by a block hsop and is therefore a complete intersection.
A transcendence basis for the field of fractions is given; in fact we show $\field[V_{m,\lambda}]^G[x_1]^{-1}$ is a
``localised polynomial algebra''. For various small dimensional cases, we give generating sets for the rings of invariants
and for the other cases we give explicit input sets for the SAGBI/Divide-by-$x$ algorithm
 introduced in \cite[\S 1]{CSW}.

The odd dimensional representations are considered in sections \ref{easy} and \ref{hard}. We show that the Noether number for
$\Omega^{-m}(\field)$ is $m+1$ (Corollary~\ref{easy-odd-cor}), the Noether number for
$\Omega^{m}(\field)$ is $3m$ for $m>1$ (Corollary~\ref{hard-odd-cor}), and that in all cases the Hilbert ideal is generated by a block hsop.
We give generating sets for $\field[\Omega^{-m}(\field)]^G[x_1^{-1}]$ and for  $\field[\Omega^{m}(\field)]^G[(x_1x_2(x_1+x_2))^{-1}]$.
We also give explicit input sets for the SAGBI/Divide-by $x$ algorithm.

\section{Preliminaries}
\label{preliminaries}

Let $\field$ denote a field of characteristic $2$.
Suppose $\langle\sigma\rangle\cong\ztwo$ acts on
$\ring:=\field[x_1,\ldots,x_m,y_1,\ldots,y_m]$
by $\sigma(x_j)=x_j$, $\sigma(y_j)=y_j+x_j$.
Define $\Delta:=\sigma-1$ and $n_i:=y_i^2+x_iy_i$.
We will often write $S^{\sigma}$ as short-hand for $S^{\langle\sigma\rangle}$.

\begin{proposition}\label{sigma-gen}
{\rm (}\cite{richman}, \cite{campbell-hughes}, \cite{CSW-ad}{\rm
)}
 $\ring^{\sigma}$ is generated by
$$\{n_1,\ldots,n_m\}\cup
\{\Delta(\beta)\mid \beta\ {\rm divides}\ y_1\cdots y_m\}.$$
\end{proposition}

\begin{corollary}\label{del-lem}
$\Delta\ring=\left(\left(x_1,\ldots,x_m\right)\ring\right)^{\sigma}$ and
$\ring^{\sigma}/\Delta\ring \cong\field[n_1,\ldots,n_m]$.
\end{corollary}
\begin{proof}
It is clear from the definition of $\Delta$ that $\Delta\ring\subset (x_1,\ldots,x_m)S$.
Since $\Delta^2=0$, we have $\Delta\ring\subseteq \left(\left(x_1,\ldots,x_m\right)S\right)^{\sigma}$.
The result then follows from the definition of $n_i$ and the generating set for $\ring^{\sigma}$ given above.
\end{proof}

Proposition~\ref{sigma-gen} and Corollary~\ref{del-lem} give the following.

\begin{lemma}\label{ilk}
Suppose $a_1, \dots ,a_m$ are non-negative integers. Let $f\in
S^\sigma$.
\begin{enumerate}
\item[(i)] If $y_1^{a_1}\cdots y_m^{a_m}$ appears in $f$, then
$a_i$ is even for $i\in\{1,\ldots,m\}$.
%$1\le i\le m$.
\item[(ii)] If  $y_1^{a_1}\cdots y_{m-1}^{a_{m-1}}y_mx_m$ appears
in $f$, then $a_i$ is even for $i\in\{1,\ldots,m-1\}$.
%$1\le i\le m-1$.
\end{enumerate}
\end{lemma}

A simple calculation shows that for $a,b\in\ring$,
$$\Delta(a\cdot b)=\Delta(a)b+a\Delta(b)+\Delta(a)\Delta(b)$$
and $\Delta(a^2)=\Delta(a)^2$. Therefore, if $M=y_1^{a_1}\cdots y_m^{a_m}$ with $a_i>0$,
then the monomial $x_iM/y_i$ appears in $\Delta(M)$ with coefficient $1$ if
$a_i$ is odd and coefficient $0$ if $a_i$ is even.
Note that if a monomial $M$ appears (with non-zero coefficient) in $f\in\ring^\sigma$
and a monomial $M'$ appears in $\Delta(M)$, then there is at least one
further monomial, say $M''$, with $M\not=M''$, such that $M''$ appears in $f$ and
$M'$ appears in $\Delta M''$.

\begin{lemma}
\label{permute} Suppose $M'$ is a monomial in $\{y_1,\dots, y_m\}$
and $M=M'x_iy_j$ for some $i,j\in\{1,\ldots,m\}$ with $i\neq j$.
Assume further that the degree of $y_j$ in $M'$ is even. If $M$
appears in a polynomial $f\in S^\sigma$, then the degree of $y_i$
in $M'$ is even and $M'x_jy_i$ also appears in $f$. Moreover, the
coefficients in $f$ of these monomials are the same.
\end{lemma}
\begin{proof}
Since the degree of $y_j$ in $M$ is odd, $M'x_ix_j$ appears in
$\Delta(M)$ with coefficient 1. Note that if the degree of $y_i$
in $M'$ is odd, then there is no other monomial in $\ring$ that
produces $M'x_ix_j$ after applying $\Delta$. Therefore, we may
assume that the degree of $y_i$ in $M'$ is even. In this case,
$M'x_ix_j$ appears in  $\Delta(M'y_ix_j)$ and $\Delta (M'y_iy_j)$.
However, the degree of $y_j$ in the monomial $M'y_iy_j$ is odd so
it follows from  Lemma~\ref{ilk} that $M'y_iy_j$ does not appear
in $f$. Therefore $M'y_ix_j$ appears in $f$. Since the coefficient
of $M'x_ix_i$ in both $\Delta(M'y_ix_j)$ and $\Delta(M'y_jx_i)$ is
$1$, the coefficients of $M'y_ix_j$ and $M'y_jx_i$ in $f$ must be
equal.
\end{proof}

\begin{lemma}
\label{lift} Suppose that $M'$ is a monomial in $\{y_1,
\dots,y_m\}\setminus \{y_j\}$ for some $j\in\{1,\ldots,m\}$ and
$M=M'y_jx_j$. For $f\in S^\sigma$, $M$ appears in $f$ if and only
if   $M'y_j^2$  appears in $f$. Moreover, the coefficients in $f$
of these monomials are the same. Finally, $M'y_j^3x_j$ does not
appear in any polynomial in $S^\sigma$.
\end{lemma}
\begin{proof}
Note that $M'x_j^2$ appears in both $\Delta(M)$ and $\Delta
(M'y_j^2)$ with coefficient $1$. Since these are the only
monomials in $S$ that produce $M'x_j^2$ after applying $\Delta$,
the result follows. The final statement follows from the fact that
$M'y_j^3x_j$ is the only monomial in $S$ that produces
$M'y_j^2x_j^2$ after applying $\Delta$.
\end{proof}

\section{Block HSOPs}
\label{block}

In this section, $G$ is an arbitrary finite group, $\field$ is a field of characteristic $p$ for some prime number $p$
 dividing the order of $G$ and $V$ is a finite dimensional $\field G$-module.
Suppose we have a homogeneous system of parameters $\mathcal S
=\{h_1,\ldots,h_n\}$ for $\field[V]^G$. Let $A$ denote the algebra
generated by $\mathcal S$ and let $I$ denote the ideal
$(h_1,\ldots,h_n)\field[V]$. Further suppose that there exists a
term order for which $\mathcal S$ is a Gr\"obner basis for $I$ and
the reduced monomials are the monomial factors of a given
monomial, say $\beta$. Then the monomial factors of $\beta$ are a
basis for $\field[V]$ as a free $A$-module; in the language of
\cite{CHSW}, we have a block basis for $\field[V]$ over $A$. In
this situation, we will refer to $\mathcal S$ as a {\it block
hsop} and $\beta$ as the {\it top class}. Note that if the
elements of $\{\lm(h_1),\ldots,\lm(h_n)\}$ are pair-wise
relatively prime then $\mathcal S$ is a block hsop and the top
class is the unique maximal reduced monomial.

\begin{theorem}\label{block-hsop-thm}
 Suppose $\mathcal S=\{h_1,\ldots,h_n\}$ is a block hsop with top class $\beta$.
If $\tr(\beta)$ is indecomposable in $\field[V]^G$, then
\begin{enumerate}
\item[(a)] the Noether number for $V$ is $\deg(\beta)$;
\item[(b)] the Hilbert ideal of $V$ is generated by $\mathcal S$.
\end{enumerate}
\end{theorem}
\begin{proof}
Proof of (a):
The indecomposability of $\tr(\beta)$ gives a lower bound on the Noether number.
The fact that $\deg(\beta)$ is also an upper bound follows from
 \cite[Corollary 0.3]{symonds}.

Proof of (b): Denote the Hilbert ideal of $V$ by $\mathfrak h$.
Since $\mathcal S\subset\field[V]^G$, we have $I\subseteq\mathfrak h$.
Suppose, by way of contradiction,  there exists $f\in\mathfrak h \setminus I$.
We may assume that $f$ is homogeneous and that $\lm(f)$ is reduced with respect to
$I$ using the chosen term order.
Therefore $\lm(f)$ divides $\beta$. Reducing $\beta$ with
respect to $\mathcal S \cup \{f\}$, produces a polynomial of degree $d:=\deg(\beta)$
with lead term less than $\beta$. However, $\field[V]/I$ in degree $d$ has dimension one.
Thus $\beta\in(h_1,\ldots,h_n,f)\field[V]\subseteq \mathfrak h$.
Let $\mathcal C$ be the reduced monomials with respect to $\mathfrak h$ using the chosen term order.
Observe that the elements of $\mathcal C$ are monomial factors of
$\beta$ with degree less than $d$. Since $\mathcal C$ generates $\field[V]$ as
an $\field[V]^G$-module, the transfer ideal, $\tr(\field[V])$, is generated
by $\{\tr(\gamma)\mid\gamma\in\mathcal C\}$ as an $\field[V]^G$-module.
Therefore, $$\tr(\beta)=\sum_{\gamma\in\mathcal C}c_\gamma\tr(\gamma)$$
for some $c_\gamma\in\field[V]^G$.
Since the representation is modular, $\tr(1)=0$. Furthermore $\deg(\tr(\gamma))<d$.
Therefore, the equation above  gives a decomposition of $\tr(\beta)$ in terms of invariants of degree less than $d$,
contradicting the indecomposability of $\tr(\beta)$.
\end{proof}

\section{Even Dimensional Representations}\label{even-sec}

In this section we consider the even dimensional representations $V_{m,\lambda}$.
For completeness, we also include a brief discussion of the regular representation in
subsection~\ref{even-examples}.
For $\lambda\in\field$, the action of $G=\langle\sigma_1,\sigma_2\rangle$ on
$\ring:=\field[V_{m,\lambda}]=\field[x_1,\ldots,x_m,y_1,\ldots,y_m]$
is given by $\sigma_i(x_j)=x_j$, $\sigma_1(y_j)=y_j+x_j$,
$\sigma_2(y_1)=y_1+\lambda x_1$ and $\sigma_2(y_j)=y_j+\lambda x_j+x_{j-1}$
for $j>1$.
Define $n_i:=y_i^2+x_iy_i$ and $u_{ij}=x_iy_j+x_jy_i$.
Then $n_i,u_{ij}\in\ring^{\sigma_1}$. A simple calculation gives
$\Delta_2(n_i)=(\lambda^2+\lambda)x_i^2+x_{i-1}^2+x_ix_{i-1}$
and $\Delta_2(u_{ij})=x_ix_{j-1}+x_{i-1}x_j$ (using the convention that $x_0=0$).
Define  $\ell:=\lfloor m/2\rfloor$ and, for $i\leq \ell$, define
$$N_i:=n_i+(\lambda^2+\lambda)\sum_{j=1}^iu_{i-j+1,i+j}
+\sum_{j=1}^{i-1}\left(u_{i-j,i+j}+u_{i-j,i+j-1}\right).$$
An explicit calculation, exploiting the fact that $\Delta_2(u_{1j})=x_1x_{j-1}$,
gives $\Delta_2(N_i)=0$. Therefore $N_i\in S^G$.
Define
$$\mathcal H:=\left\{x_1,\ldots,x_m\right\}\cup\left\{\shsop i \mid 1\leq i\leq m/2\right\}
\cup\left\{\norm{y_j}\mid m/2<j\leq m\right\}.$$

\begin{theorem}\label{hsop_thm}
 $\mathcal{H}$ is a block hsop with top class $y_1\cdots y_{\ell}y_{\ell+1}^3\cdots y_m^3$.
\end{theorem}
\begin{proof}
This follows from the fact that $\lt(N_i)=y_i^2$ and $\lt(\norm{y_j})=y_j^4$.
\end{proof}

\begin{corollary} The image of the transfer, $\tr(S)$, is the ideal in $S^G$ generated by
$$\left\{\tr(\beta)\mid \beta\ {\rm divides}\ y_1\cdots y_{\ell}(y_{\ell+1}\cdots y_m)^3\right\}.$$
\end{corollary}

\begin{theorem}\label{even-indec}
For $\lambda \not\in\field_2$ and $m\geq 3$,
$\tr(y_1\cdots y_{\ell}y_{\ell+1}^3\cdots y_m^3)$ is indecomposable.
\end{theorem}

See subsection~\ref{even-indec-proof} for the proof of Theorem~\ref{even-indec}.
Combining Theorem~\ref{even-indec} with Theorem~\ref{block-hsop-thm} gives the following.

\begin{corollary} If $\lambda \not\in\field_2$ and $m\geq 3$, then
the Noether number for $V_{m,\lambda}$ is $3m-2\lfloor m/2\rfloor$ and the Hilbert ideal
is generated by $\mathcal H$.
\end{corollary}

Descriptions of $\ring^G$ for $m\leq 3$ are given in subsection~\ref{even-examples}.
The formula given above for the Noether number is valid for $m>1$.

For $j>1$, an explicit calculation gives
\begin{eqnarray*}
\tr(y_1y_2y_j)&=&y_1(x_2x_{j-1}+x_1x_j)+y_2x_1x_{j-1}+y_jx_1^2\\&&
+x_1x_2\left((\lambda^2+\lambda)x_j+x_{j-1}\right)+x_1^2(x_j+x_{j-1})\\
&=&u_{12}x_{j-1}+u_{1j}x_1
+\tr(y_1y_3)\left((\lambda^2+\lambda)x_j+x_{j-1}\right)\\ &&+\tr(y_1y_2)(x_j+x_{j-1}).
\end{eqnarray*}
Therefore $t_j:=u_{12}x_{j-1}+u_{1j}x_1\in\tr(S)$.

\begin{theorem}\label{field_thm} For $m>3$ and $\lambda\not\in\field_2$,
$$\field[V_{m,\lambda}]^G[x_1^{-1}]=\field[x_1,\ldots,x_m,N_1,N_2,t_3,\ldots,t_m][x_1^{-1}].$$
\end{theorem}
\begin{proof}
We use \cite[Theorem~2.4]{CampChuai}. $\field[x_1,\ldots,x_m,y_1]^G$ is the polynomial algebra generated by
$\{x_1,\ldots,x_m,N_G(y_1)\}$. Since
$N_1=y_1^2+x_1y_1+(\lambda^2+\lambda)(x_1y_2+x_2y_1)$, we see that $N_1\in\field[x_1,x_2,y_1,y_2]$ is degree $1$ in $y_2$
with coefficient $(\lambda^2+\lambda)x_1$. Using the equation above, $t_j\in\field[x_1,\ldots,x_m,y_1,y_2,y_j]$
is degree $1$ in $y_j$ with coefficient $x_1^2$. Thus
$\ring^G[x_1^{-1}]=\field[x_1,\ldots,x_m,N_G(y_1),N_1,t_3,\ldots,t_m][x_1^{-1}]$.
To complete the proof,
we need only rewrite $N_G(y_1)$  in terms of $N_2$ and the other generators. An explicit calculation
gives $$N_G(y_1)=y_1^4+x_1^2y_1^2(\lambda^2+\lambda+1)+x_1^3y_1(\lambda^2+\lambda).$$
Define $c:=\lambda^2+\lambda$.
Subduction gives
$$\norm{y_1}=N_1^2+((cx_2)^2+cx_1^2)N_1+(cx_1)^2N_2+(c^3x_2+c^2x_1)t_3
+c^3x_1t_4,$$
as required.
\end{proof}

%The above shows that for $m=3$,
%$\ring^G[x_1^{-1}]=\field[x_1,x_2,x_3,N_1,\norm{y_1},t_3][x_1^{-1}].$

\begin{remark} For $m>3$ and $\lambda\not\in\field_2$, it follows from Theorem~\ref{field_thm} and Theorem~\ref{hsop_thm},
that $\ring^G$ is the normalisation of the algebra generated by
$\mathcal{B}:=\mathcal{H}\cup\{t_3,\ldots,t_m\}$. Furthermore, applying the
SAGBI/Divide-by-$x$ algorithm of \cite{CSW} with $x=x_1$ to $\mathcal{B}$ computes a SAGBI basis for
$\ring^G$.
\end{remark}

Using the familiar formula for the group cohomology of a cyclic group,
we have
$$H^1\left(\langle\sigma_2\rangle,\Delta_1 S\right)\cong\left(\Delta_1 S\right)^{\sigma_2}/\Delta_2\Delta_1 S
=\left(\Delta_1 S\right)^{\sigma_2}/\tr S$$
and
$H^1\left(\langle\sigma_1\rangle,\Delta_2 S\right)\cong\left(\Delta_2 S\right)^{\sigma_1}/\tr S.$
Note that $H^1\left(\langle\sigma_1\rangle,\Delta_2 S\right)$ and $H^1\left(\langle\sigma_2\rangle,\Delta_1 S\right)$
are both finitely generated $S^G$-modules and, therefore,
are also finitely generated over the algebra generated by $\mathcal H$.
In the following $\sqrt{\tr\ring}$ denotes the radical of the image of the transfer.

\begin{proposition}\label{even-coh-prop}
For $\lambda\not\in\field_2$, $(\Delta_2\ring)^{\sigma_1}=(\Delta_1S)^{\sigma_2}
=\left(\left(x_1,\ldots,x_m\right)\ring\right)^G=\sqrt{\tr\ring }$ and
$$\sqrt{\tr \ring}/\tr \ring\cong H^1(\langle\sigma_2\rangle,\Delta_1\ring)\cong H^1(\langle\sigma_1\rangle,\Delta_2\ring).$$
Furthermore $S^G/\sqrt{\tr S}\cong \field[N_1,\ldots,N_{\ell},\norm{y_{\ell+1}},\ldots,\norm{y_m}]$.
\end{proposition}
\begin{proof} For $\lambda\not\in\field_2$,
$$\Delta_1 V_{m,\lambda}^*=\Delta_2 V_{m,\lambda}^*=(\sigma_1\sigma_2+1)V_{m,\lambda}^*={\rm Span}_{\field}\{x_1,\ldots,x_m\}.$$
Using
\cite[Theorem~2.4]{modular-transfer} (see also \cite[Theorem~2.4]{feshbach}),
$\sqrt{\tr\ring}=\left((x_1,\ldots,x_m)\ring\right)^G$.
Applying Corollary~\ref{sigma-gen} with $\sigma=\sigma_1$ gives
$\Delta_1\ring=\left(\left(x_1,\ldots,x_m\right)\ring\right)^{\sigma_1}$.
Thus $(\Delta_1\ring)^{\sigma_2}=\left((x_1,\ldots,x_m)\ring\right)^G$.
Applying the Corollary~\ref{sigma-gen} with $\sigma=\sigma_2$ gives
 $(\Delta_2\ring)^{\sigma_1}=\left((x_1,\ldots,x_m)\ring\right)^G$.

To prove the final statement, first observe that
$$\mathcal N:=\{N_1,\ldots,N_{\ell},\norm{y_{\ell+1}},\ldots,\norm{y_m}\}$$ is algebraically independent modulo
$\sqrt{\tr\ring}$. Therefore, there is a subalgebra of  $S^G/\sqrt{\tr S}$ isomorphic
to $A:= \field[N_1,\ldots,N_{\ell},\norm{y_{\ell+1}},\ldots,\norm{y_m}]$. We will show that for every
$f\in\ring^G$, there exists $F\in A$ with $f-F\in\sqrt{\tr S}$.
We proceed by minimal counter-example.
Without loss of generality, we may assume $f$ is homogeneous of positive degree.
Since $\lm(g(y_i))=y_i$ for all $g\in G$, using \cite[Theorem~3.2]{sw-comp}, there exists a finite SAGBI basis for $S^G$
and therefore a finite SAGBI-Gr\"obner basis for the ideal $\sqrt{\tr S}$.
We may assume that $f$ is reduced, i.e., equal to its normal form.
Therefore $\lm(f)=\prod_{i=1}^my^{a_i}$. Using Lemma~\ref{ilk}, each $a_i$ is even.
It follows from  Proposition~\ref{double} that $\lm(f)$ does not divide $\prod_{i=\ell+1}^my_i^2$.
Since $\lt(N_i)=y_i^2$ and $\lt(\norm{y_j})=y_j^4$,
there exits $N\in\mathcal N$  with $\lt(N)=y_k^{b_k}$ dividing $\lm(f)$.
Note that $N=y_k^{b_k}+\widetilde{N}$ for some $\widetilde{N}\in(x_1,\ldots,x_m)\ring$.
Since $N$ is monic as a polynomial in $y_k$, we can divide $f$ by $N$ to get
$f=qN+r$ for unique $q,r\in\ring$ with $\deg_{y_k}(r)< \deg_{y_k}(N)=b_k$. Furthermore, since we are using
grevlex with $x_i<y_k$, we have $\lm(r)<\lm(f)$. Applying $g\in G$, gives
$f=g(f)=g(q)N+g(r)$. However, $\deg_{y_k}(g(r))\leq\deg_{y_k}(r)$. Therefore, by the uniqueness of the remainder,
$g(r)=r$ and $g(q)=q$. Thus $q,r\in\ring^G$ with $q<f$ and $r<f$. By the minimality of $f$, there exists
$F_1,F_2\in A$ with $q-F_1,r-F_2\in\sqrt{\tr S}$. Therefore $F:=NF_1-F_2\in A$ and $f-F\in\sqrt{\tr S}$,
giving the required contradiction.
\end{proof}

While $V_{m,0}$ and $V_{m,1}$ are not equivalent representations, the automorphism of $G$ which fixes $\sigma_1$ and exchanges
$\sigma_2$ and $\sigma_1\sigma_2$, takes $V_{m,0}$ to $V_{m,1}$. Therefore $\field[V_{m,0}]^G\cong\field[V_{m,1}]^G$. Hence, to compute
 $\field[V_{m,\lambda}]^G$ with $\lambda\in\field_2$, it is sufficient to take $\lambda=0$.

Substituting $\lambda=0$ into the expression for $N_i$ given above, gives an element in $\field[V_{m,0}]^G$ with lead term
$y_i^2$ for $i\leq \lceil m/2\rceil$. Define $\ell':= \lceil m/2\rceil$ and
$$\mathcal H':=\left\{x_1,\ldots,x_m\right\}\cup\left\{\shsop i \mid 1\leq i\leq (m+1)/2\right\}
\cup\left\{\norm{y_j}\mid (m+1)/2< j\leq m\right\}.$$
Looking at lead terms gives the following.

\begin{theorem}\label{zero-hsop_thm}
 For $\lambda\in\field_2$, $\mathcal{H}'$ is a block hsop with top class $y_1\cdots y_{\ell'}y_{\ell'+1}^3\cdots y_m^3$.
\end{theorem}

\begin{theorem}\label{zero-even-indec}
For $\lambda \in\field_2$ and $m>3$,
$\tr(y_1\cdots y_{\ell'}y_{\ell'+1}^3\cdots y_m^3)$ is indecomposable.
\end{theorem}

See subsection~\ref{zero-even-indec-proof} for the proof of Theorem~\ref{zero-even-indec}.
Combining Theorem~\ref{zero-even-indec} with Theorem~\ref{block-hsop-thm} gives the following.

\begin{corollary} For $m>3$,
the Noether number for $V_{m,0}$ is $3m-2\lceil m/2\rceil$ and the Hilbert ideal
is generated by $\mathcal{H}'$.
%$\{x_1,\ldots,x_m,N_1,\ldots,N_{\ell'},\norm{y_{\ell'+1}},\ldots,\norm{y_m}\}$.
\end{corollary}

Descriptions of $\field[V_{m,0}]^G$ for $m\leq 3$ are given in subsection~\ref{even-examples}.
The above formula for the Noether number is valid for $m>1$.

\begin{theorem}\label{zero-field_thm} For $m>2$,
$$\field[V_{m,0}]^G[x_1^{-1}]=\field[x_1,\ldots,x_m,N_1,N_2,t_3,\ldots,t_m][x_1^{-1}].$$
\end{theorem}
\begin{proof}
We construct the field of fractions for an upper-triangular action as in \cite{CampChuai}
or \cite{Kang}. From  Remark~\ref{zero-two_remark}, we see that
$\field[x_1,x_2,y_1,y_2]^G[x_1^{-1}]=\field[x_1,x_2,N_1,\widetilde{w}][x_1^{-1}]$
where  $\widetilde{w}:=(x_1+x_2)u_{12}+x_1n_2$. Since
$t_j\in\field[x_1,\ldots,x_m,y_1,\ldots,y_j]^G$
has degree one as a polynomial in $y_j$ with coefficient $x_1^2$, we have
$$\field[V_{m,0}]^G[x_1^{-1}]=\field[x_1,\ldots,x_m,N_1,\widetilde{w},t_3,\ldots,t_m][x_1^{-1}].$$
The result then follows from the relation $\widetilde{w}=x_1N_2+t_3$.
\end{proof}

\begin{remark} For $m>2$ it follows from Theorem~\ref{zero-field_thm} and Theorem~\ref{zero-hsop_thm},
that $\field[V_{m,0}]^G$ is the normalisation of the algebra generated by
$\mathcal{B}':=\mathcal{H}'\cup\{t_3,\ldots,t_m\}$. Furthermore, applying the
SAGBI/Divide-by-$x$ algorithm of \cite{CSW} with $x=x_1$ to $\mathcal{B}'$ computes a SAGBI basis for
$\field[V_{m,0}]^G$.
\end{remark}

\begin{proposition} For $\lambda=0$:
\begin{eqnarray*}
\sqrt{\tr\ring}&=&\left(\left(x_1,\ldots,x_{m-1}\right)\ring\right)^G,\cr
H^1(\langle\sigma_1\rangle,\Delta_2\ring)&\cong&\left(\left(x_1,\ldots,x_{m-1}\right)\ring\right)^G/\tr\ring,\cr
H^1(\langle\sigma_2\rangle,\Delta_1\ring)&\cong&\left(\left(x_1,\ldots,x_{m}\right)\ring\right)^G/\tr\ring,\cr
\ring^G/\left(\left(x_1,\ldots,x_{m}\right)\ring\right)^G
&\cong& \field[N_1,\ldots,N_{\ell'},\norm{y_{\ell'+1}},\ldots,\norm{y_m}].
\end{eqnarray*}
\end{proposition}
\begin{proof} Direct calculation gives
$\Delta_1 V_{m,0}^*=(\sigma_1\sigma_2+1) V_{m,0}^*={\rm Span}_{\field}\{x_1,\ldots,x_m\}$ and
$\Delta_2 V_{m,0}^*={\rm Span}_{\field}\{x_1,\ldots,x_{m-1}\}$.
Using \cite[Theorem~2.4]{modular-transfer},
$$\sqrt{\tr\ring}=\bigcap_{g\in G,\,|g|=2}(((g-1)V_{m,0}^*)S)^G=
\left(\left(x_1,\ldots,x_{m-1}\right)\ring\right)^G.$$
The rest of the proof is analogous to the proof of Proposition~\ref{even-coh-prop}.
\end{proof}

\subsection{\rm Even Dimensional Examples}\label{even-examples}

\begin{sremark}
It follows from \cite[Theorem~3.75]{DK}, that
$\field[V_{1,\lambda}]^G$ is the polynomial ring generated by $x_1$ and $\norm{y_1}$.
\end{sremark}

Define $w:=\Delta_2(n_2)u_{12}+x_1^2n_2$. Note that $\norm{y_2}=n_2^2+n_2\Delta_2(n_2)$
and recall that $\Delta_2(n_2)=(\lambda^2+\lambda)x_2^2+x_1x_2+x_1^2$.
A simple calculation shows that $\lt(w)=(\lambda^2+\lambda)y_1x_2^3$.
Subduction gives
\begin{equation}\label{two-lambda-rel}
w^2=\Delta_2(n_2)^2x_2^2N_1+x_1^4\norm{y_2}+w\Delta_2(n_2)\left(\Delta_2(n_2)+x_1^2\right).
\end{equation}
\begin{stheorem}\label{two-lambda-thm} If $\lambda\not\in\field_2$, then $\field[V_{2,\lambda}]^G$ is the hypersurface generated by
$x_1$, $x_2$, $N_1$, $w$ and $\norm{y_2}$, subject to the above relation.
\end{stheorem}
\begin{proof} Since $N_1$ has degree $1$ in $y_2$ with coefficient $(\lambda^2+\lambda)x_1^2$,
using \cite[Theorem~2.4]{CampChuai}, we have
$\field[V_{2,\lambda}]^G[x_1^{-1}]=\field[x_1,x_2,\norm{y_1},N_1][x_1^{-1}]$.
Subduction gives
$$\norm{y_1}=N_1^2+(\lambda^2+\lambda)^2(x_2^2N_1+w)+x_1^2(w^2+w)N_1.$$
Therefore $\field[V_{2,\lambda}]^G[x_1^{-1}]=\field[x_1,x_2,N_1,w][x_1^{-1}]$.
Furthermore $\{x_1,x_2,N_1,\norm{y_2}\}$ is a block hsop. Taking
$\mathcal B :=\{x_1,x_2,N_1,w,\norm{y_2}\}$, we see that there is a single
non-trivial \tat, which subducts to $0$ using Equation~\ref{two-lambda-rel}.
Therefore, using \cite[Theorem~1.1]{CSW}, $\mathcal B$ is a SAGBI basis for
$\field[V_{2,\lambda}]^G$.
\end{proof}

It follows from Theorem~\ref{two-lambda-thm} that the Noether number for $V_{2,\lambda}$ is $4$
and the Hilbert ideal is generated by $\{x_1,x_2,N_1,\norm{y_2}\}$.

\begin{sremark}\label{zero-two_remark}
 A Magma \cite{magma} calculation
shows that $\field[V_{2,0}]^G$ is a hypersurface with
generators $x_1, x_2, n_1, \widetilde{w}:=(x_1+x_2)u_{12}+x_1n_2, \widetilde{N}_2:=n_2^2+n_2(x_1^2+x_1x_2)$ and relation
$\widetilde{w}^2+x_2^2(x_2+x_1)^2n_1+x_1x_2(x_1+x_2)\widetilde{w}=x_1^2\widetilde{N}_2$.
Therefore the Noether number for $V_{2,0}$ is $4$ and the Hilbert ideal is generated by $x_1,x_2, n_1, \widetilde{N}_2$.
Using the relation to eliminate
$\widetilde{N}_2$ gives $\field[V_{2,0}]^G[x_1^{-1}]=\field[x_1,x_2,n_1,\widetilde{w}][x_1^{-1}]$.
\end{sremark}

Define $u_{123}:=x_1(n_2+u_{12}+u_{13})+(\lambda^2+\lambda)x_2u_{13}$. Simple calculations give
$\lm(u_{123})=y_1x_2x_3$ and $\Delta_2(u_{123})=0$.

\begin{stheorem}  If $\lambda\not\in\field_2$, then
$\field[V_{3,\lambda}]^G[x_1^{-1}]=\field[x_1,x_2,x_3,N_1,u_{123},t_3][x_1^{-1}]$.
\end{stheorem}
\begin{proof} From the proof of Theorem~\ref{two-lambda-thm},
 $\field[V_{2,\lambda}]^G[x_1^{-1}]=\field[x_1,x_2,N_1,w][x_1^{-1}]$.
Since $t_3$ is degree $1$ in $y_3$ with coefficient $x_1^2$, using \cite[Theorem~2.4]{CampChuai}, we have
$$\field[V_{3,\lambda}]^G[x_1^{-1}]=\field[x_1,x_2,x_3,N_1,w,t_3][x_1^{-1}].$$
An explicit calculation gives
$w=(\lambda^2+\lambda)x_2t_3+x_1u_{123}+x_1t_3$, and the result follows.
\end{proof}

With $c:=\lambda^2+\lambda$, define
$$n_{23}:=\left(n_2+u_{12}+u_{13}\right)\left(cx_3+x_2+x_1\right)+c\left(x_1n_3+x_2u_{23}+cx_3u_{23}\right),$$
$u_{133}:=x_1^{-1}(cx_3t_3+x_2u_{123})$,
$u_{2333}:=x_1^{-1}((cx_3+x_2)n_{222}+n_{23}x_2^2+x_2^2(u_{123}+t_3))$ and
$n_{222}:=x_1^{-2}(t_3^2+N_1(x_2^4+x_1^2x_3^2)+(c(x_2^3+x_1x_2x_3)+x_1x_2^2)t_3)$.

Straight-forward calculation gives $n_{23},u_{133},n_{222},u_{2333}\in\field[V_{3,\lambda}]^G$
and $\lt(n_{23})=cy_2^2x_3$, $\lt(u_{133})=cy_1x_3^2$, $\lt(n_{222})=y_2^2x_2^2$, $\lt(u_{2333})=c^2y_2x_3^3$.
Define
\begin{eqnarray*}
\mathcal{B}_{3,\lambda} &:=& \left\{x_1,x_2,x_3,N_1,t_3,u_{123},u_{133},n_{23},n_{222},u_{2333},N_G(y_2),N_G(y_3)\right\}\cr
&&\cup\ \left\{
\tr(y_1y_2y_3^3), \tr(y_1y_2^3y_3),\tr(y_2^3y_3^3),\tr(y_1y_2^3y_3^3)\right\}.
\end{eqnarray*}
Further calculation gives
$\lt(\tr(y_1y_2y_3^3))=cy_2y_1x_3^3$,
$\lt( \tr(y_1y_2^3y_3))=y_2^2y_1x_2^2$,
$\lt(\tr(y_2^3y_3^3))=cy_2^3x_3^3$,
$\lt(\tr(y_1y_2^3y_3^3))=cy_1y_2^3x_3^3$.

\begin{sremark} Suppose $\lambda\not\in\field_2$, i.e., $c\not=0$.
Applying the SAGBI/Divide-by-$x$ algorithm to $\{x_1,x_2,x_3,N_1, u_{123},t_3,\norm{y_2},\norm{y_3}\}$
produces a SAGBI basis for $\field[V_{3,\lambda}]^G$. A Magma calculation over the rational function field $\field_2(\lambda)$
shows that for generic $\lambda$, $\mathcal{B}_{3,\lambda}$ is a SAGBI basis for $\field_2(\lambda)[V_{3,\lambda}]^G$.
Since the lead coefficients of the elements of
 $\mathcal{B}_{3,\lambda}$ lie in $\{1,c,c^2\}$, the calculations could have been performed over
$\field_2[\lambda,c^{-1}]$. Therefore $\mathcal{B}_{3,\lambda}$ is a SAGBI basis for $\field[V_{3,\lambda}]^G$,
as long as $c\not=0$.
It follows from this that, for $\lambda\not\in\field_2$, the Hilbert ideal is generated by $x_1,x_2,x_3,N_1,\norm{y_2},\norm{y_3}$.
Although a SAGBI basis need not be a minimal generating set, running a SAGBI basis test on
$\mathcal B_{3,\lambda}\setminus\{\tr(y_1y_2^3y_3^3)\}$ shows that $\tr(y_1y_2^3y_3^3)$ is indecomposable
and hence the Noether number is $7$.
\end{sremark}

\begin{sremark}
A Magma calculation shows that $\field[V_{3,0}]^G$ is generated by
$$\{x_1,x_2,x_3,n_1,n_2+u_{13}+u_{12},t_3,(x_3+x_2)u_{13}+n_3x_1, \norm{y_3},\tr(y_2y_3^3),\tr(y_1y_2y_3^3)\}.$$
Furthermore, this is a SAGBI basis and $\tr(y_1y_2y_3^3)$ is indecomposable. Therefore the Hilbert ideal is generated
by $\{x_1,x_2,x_3,n_1,n_2+u_{13}+u_{12},\norm{y_3}\}$ and the Noether number is $5$.
\end{sremark}

The ring of invariants for the regular representation was computed in \cite[Corollary~1.8]{adem-milgram} and
\cite[Lemma~5.2]{El-Fl}. We include an alternate calculation here for completeness. Choose a basis $\{x,y_1,y_2,z\}$
for $V_{reg}^*$ so that $\Delta_i(z)=y_i$ and $\tr(z)=x$. Define $u:=y_1y_2+xz$ and
$h:=(u^2+\norm{y_1}\norm{y_2})/x=y_1^2y_2+y_2^2y_1+x(z^2+y_1y_2)$.

\begin{stheorem} $\field[V_{reg}]^G$ is the complete intersection generated by
$$\mathcal C=\{x,u,\norm{y_1},\norm{y_2},h,\norm{z}\}$$ subject to the relations
$$ u^2=\norm{y_1}\norm{y_2}+xh$$
and
$$h^2=\norm{y_1}^2\norm{y_2}+\norm{y_1}\norm{y_2}^2+x\left(h\norm{y1}+uh+h\norm{y_2}+x\norm{z}\right).$$
\end{stheorem}
\begin{proof} It follows from \cite[Theorem~3.75]{DK}, that
$\field[x,y_1,y_2]^G$ is the polynomial ring generated by $x$, $\norm{y_1}$ and $\norm{y_2}$.
Since $u$ is degree $1$ in $z$ with coefficient $x$, using \cite[Theorem~2.4]{CampChuai}, we have
$\field[V_{reg}]^G[x^{-1}]=\field[x,\norm{y_1},\norm{y_2},u][x^{-1}]$. Using the graded reverse lexicographic order
with $z>y_1>y_2>x$, there are two non-trivial \tat s among the elements of $\mathcal C$.
These two \tat s subduct to zero using the given relations. Therefore $\mathcal C$ is a SAGBI basis for the subalgebra it generates.
Since $\{x,\norm{y_1},\norm{y_2},\norm{z}\}$
is a block hsop, applying \cite[Theorem~1.1]{CSW} shows that $\mathcal C$ is a SAGBI basis for $\field[V_{reg}]^G$.
Since all relations come from subducting \tat s, the ring of invariants is the given complete intersection.
\end{proof}

It follows from the above theorem that  for $V_{reg}$ the Noether number is $4$ and the Hilbert ideal is generated
by $\{x,u,\norm{y_1},\norm{y_2},\norm{z}\}$. We note that $V_{reg}$ is the only indecomposable modular representation
of $G$ whose
Hilbert ideal is not generated by a block hsop.

\medskip

\subsection{\rm The Proof of Theorem~\ref{even-indec}}\label{even-indec-proof}
\hfil

\medskip
Suppose, by way of contradiction, that $\tr(y_1\cdots y_{\ell}y_{\ell+1}^3\cdots y_m^3)$ is
decomposable. Working modulo the $G$-stable ideal
$(x_1,\ldots,x_{m-1})S$, it easy to see that
$$\lt(\tr(y_1\cdots  y_{\ell}y_{\ell+1}^3\cdots y_m^3))=(\lambda^2+\lambda) y_1\cdots y_{\ell}y_{\ell+1}^3\cdots y_{m-1}^3x_m^3.$$
Thus there are two monomials of positive degree, say $M_1$ and $M_2$, such that
$M_1M_2=y_1\cdots y_{\ell}y_{\ell+1}^3\cdots y_{m-1}^3 x_m^3$,
and both $M_1$ and $M_2$ appear in $G$-invariant polynomials.
We use the following results to rule out possible factorisations.

\begin{slemma} \label{shift}
Suppose  $f\in\ring^G$, $M'$ is a monomial in $y_1,\ldots,y_m$, and $i>1$.
If the degree of $y_i$ in $M'$ is even then $M'y_ix_m$
does not appear in $f$.
%a $G$-invariant polynomial.
Further suppose $j<m$.
Then the degree of $y_i$ in $M'$ is even and $M'y_ix_j$ appears in $f$ if and only if
 the degree of $y_{j+1}$ in $M'$ is even and $M'y_{j+1}x_{i-1}$ appears in $f$.
\end{slemma}

\begin{proof}
We list the monomials in $S$ that produce $M'x_{i-1}x_j$ after
applying $\del_2$:
\begin{enumerate}
\item $M'y_ix_j$ if the degree of $y_i$ in $M'$ is even;
\item $M'x_{i-1}y_{j+1}$ if $j<m$ and the degree of $y_{j+1}$ in
$M'$ is even;
\item $M'x_{i-1}y_{j}$ if the degree of $y_{j}$ in
$M'$ is even and $\lambda \neq 0 $;
 \item $M'y_{i-1}x_{j}$
 if the degree of $y_{i-1}$ in $M'$ is even and
$\lambda \neq 0$;
 \item $M'y_{i-1}y_{j}$ if the degree of
$y_{i-1}$ and $y_j$ in $M'$ is even and $\lambda\neq 0$;
\item $M'y_{i-1}y_{j+1}$ if  $j<m$ and the degree of $y_{i-1}$ and
$y_{j+1}$ in $M'$ is even and $\lambda\neq 0$;
\item $M'y_{i}y_{j+1}$ if $j<m$ and the degree of $y_{i}$ and $y_{j+1}$
in $M'$ is even;
 \item $M'y_{i}y_{j}$ if $i\neq j$ and  the degree
of $y_{i}$ and $y_{j}$ in $M'$ is even and $\lambda\neq 0$.
\end{enumerate}
Note that the monomials in (5)--(8) do not appear in $f$ by Lemma~\ref{ilk}
because the degree of either $y_i$ or $y_{i-1}$
%in these monomials
is odd. On the other hand, by Lemma~\ref{permute} the
monomials in (3) and (4) appear in $f$ with the
same coefficient (which is possibly zero). Call this coefficient
$\alpha$.  Then the coefficient of $M'x_{i-1}x_j$ in
$\del_2(\alpha M'x_{i-1}y_{j}+\alpha M'y_{i-1}x_{j})$
 is $2\lambda\alpha=0$. It follows that  the monomial in (1)
 appears in $f$ if and only if the monomial in (2) appears in $f$.
% since $\del_2(f)=0$.
 \end{proof}

\begin{sproposition}
\label{double} Let $M=\prod_{i\in I}y_i^2$  for some non-empty
subset $I\subseteq \{1, \dots , m\}$ and assume that $M$ appears
in a polynomial $f\in S^G$. Let $j$ denote the maximum integer in
$I$. Then  $2j\le m+1$.
Furthermore, if $\lambda \in \field \setminus \field_2 $ then $2j\le m$.
\end{sproposition}
\begin{proof}
If $j=1$, then $2j\leq m+1$ implies $m\geq 1$ and $2j\leq m$ gives $m>1$.
For $m=1$, we have $S^G=\field [x_1, \norm{y}]$ and, if $\lambda\in\field\setminus\field_2$, then
$\lt(\norm{y_1})=y_1^4$. Thus the assertion holds for $j=1$.

Suppose $j>1$ and assume that $M$ is maximal among all such monomials that appear in $f$.
Let $M'$ denote the monomial $\prod_{i\in I\setminus \{j\}}y_i^2$.
Using Lemma~\ref{lift} (with $\sigma=\sigma_1)$, we see that $M'x_jy_j$  appears in $f$.
Since $j>1$, by Lemma~\ref{shift}, $j<m$ and $M'x_{j-1}y_{j+1}$ appears in $f$.
Applying Lemma~\ref{permute} shows that
$M'x_{j+1}y_{j-1}$ appears in $f$. If $j-1>1$, then, again using Lemma~\ref{shift}, we have
$j+1<m$ and $M'x_{j-2}y_{j+2}$ appears in $f$.
In this case, by applying Lemma~\ref{permute}, we see that
$M'x_{j+2}y_{j-2}$ appears in $f$.
Continue alternating  Lemma~\ref{shift} and Lemma~\ref{permute} until $j-k=1$. This shows that
$M'y_{j-k}x_{j+k}=M'y_1x_{2j-1}$ appears in $f$. Thus $2j-1\leq m$, as required.

Suppose that $\lambda \in \field \setminus \field_2 $.
Note that $M'x_j^2$ appears in $\del_2( M+M'x_jy_j)$ with
coefficient $\lambda+\lambda^2\neq 0$. Since $\del_2 (f)=0$, there
must be other monomials in $f$ that produce $M'x_j^2$ after
applying $\del_2$. The monomials $M'y_jy_{j+1}$, $M'x_jy_{j+1}$
and $M'y_{j+1}^2$ are the only such monomials.
% in $S$ with this property.
However, $M'y_jy_{j+1}$ does not appear in $f$ by Lemma
\ref{ilk} and the  maximality of $j$ implies that $M'y_{j+1}^2$ doe not
appear in $f$ either. It follows that $M'x_jy_{j+1}$ appears in
$f$. Applying Lemma~\ref{permute} and Lemma~\ref{shift}
repeatedly we see that $M'x_1y_{2j}$ appears in $f$.
Hence $2j\le m$.
\end{proof}

Write
$M_1=y_1^{a_1}\cdots y_{m-1}^{a_{m-1}}x_m^{a_m}$ and
$M_2=y_1^{b_1}\cdots y_{m-1}^{b_{m-1}}x_m^{b_m},$ where $a_i$ and $b_i$
are non-negative integers. % for $1\le i\le m$.
We have $a_i+b_i=1$ for $i\le \ell$ and $a_i+b_i=3$ for $i>\ell$.

Suppose $a_m=0$. Then, using Lemma~\ref{ilk} (with $\sigma=\sigma_1$), $a_i$ is even for all $i$.
Thus $a_i=0$ for $i\leq\ell$. Hence Proposition~\ref{double} applies, forcing $a_i=0$ for
$i>\ell\geq m/2$. Therefore, if $a_m=0$, we have $M_1=1$ and the factorisation is trivial.
Hence $a_m>0$. Similarly, $b_m>0$. Without loss of generality, we assume $a_m=1$ and $b_m=2$.

\begin{slemma}\label{even-a}
 If $m\ge 3$, then   $a_{m-1}$ is even. If $m\ge 4$, then   $a_{m-2}$ is even.
 \end{slemma}
\begin{proof}
Both statements follow from Lemma~\ref{shift}.

\end{proof}

\begin{slemma}\label{parity}
If $m\ge 3$, then $b_{m-1}$ and $b_{m-2}$ are not both odd.
\end{slemma}
\begin{proof}
Assume on the contrary that both $b_{m-1}$ and $b_{m-2}$ are odd
and that $M_2$ appears in $f_2\in S^G$.
Define $M=y_1^{b_1}\cdots
 y_{m-3}^{b_{m-3}}y_{m-2}^{b_{m-2}-1}y_{m-1}^{b_{m-1}-1}$ so that
 $M_2=My_{m-2}y_{m-1}x_m^2$. Then $Mx_{m-2}y_{m-1}x_m^2$ appears
 in $\del_1 (My_{m-2}y_{m-1}x_m^2)$. Since $\del_1(f_2)=0$, there
 must be other monomials in $f_2$ that produce
 $Mx_{m-2}y_{m-1}x_m^2$ after applying $\del_1$. The only monomials
with this property are $My_{m-2}y_{m-1}y_m^2$,
 $My_{m-2}y_{m-1}x_my_m$, $Mx_{m-2}y_{m-1}y_m^2$ and
 $Mx_{m-2}y_{m-1}x_my_m$. However $My_{m-2}y_{m-1}y_m^2$ does not
 appear in $f_2$ by Lemma \ref{ilk} because the degree of
 $y_{m-1}$ in this monomial is odd. Also, $My_{m-2}y_{m-1}x_my_m$
 does not appear in $f_2$ by Lemma \ref{shift}. If
 $Mx_{m-2}y_{m-1}y_m^2$ appears in $f_2$, then, since the degree of
 $y_{m-2}$ in this monomial is odd,
 $Mx_{m-2}^2y_m^2$ appears in $\del_2(Mx_{m-2}y_{m-1}y_m^2)$. So
 there must be another monomial in $f_2$ that produces
 $Mx_{m-2}^2y_m^2$ after applying $\del_2$. The only monomials in
 $S$ with this property are $My_{m-1}^2y_m^2$ if $b_{m-1}=1$, $My_{m-2}^2y_m^2$ if $b_{m-2}=1$,
  $My_{m-2}y_{m-1}y_m^2$ and
 $Mx_{m-2}y_{m-2}y_m^2$. The first three monomials do not appear in
 $f_2$ by Lemma~\ref{ilk} and Proposition~\ref{double}. On the
 other hand $Mx_{m-2}y_{m-2}y_m^2$ does not appear in $f_2$ if
 $b_{m-2}=3$ by Lemma~\ref{shift}. If $b_{m-2}=1$, then
 $Mx_{m-2}y_{m-2}y_m^2$ appears in $f_2$ if and only if
 $My_{m-2}^2y_m^2$ appears in $f_2$. However the latter monomial does
 not appear in $f_2$ by Lemma~\ref{ilk} and Proposition~\ref{double}.
Therefore  $Mx_{m-2}y_{m-1}y_m^2$ doe not appear in $f_2$.

 We finish the proof by showing that  $Mx_{m-2}y_{m-1}x_my_m$ does
 not appear in $f_2$.
 Note that
$Mx_{m-2}^2x_my_m$ appears in
 $\del_2(Mx_{m-2}y_{m-1}x_my_m)$. The other monomials that
 produce $Mx_{m-2}^2x_my_m$ after applying $\del_2$ are
 $My_{m-1}^2x_my_m$ if $b_{m-1}=1$, $My_{m-2}^2x_my_m$ if
 $b_{m-2}=1$, $My_{m-2}y_{m-1}x_my_m$ and $Mx_{m-2}y_{m-2}x_my_m$.
 The first two  monomials appear in $f_2$ if and only if $My_{m-1}^2y_m^2$
 and $My_{m-2}^2y_m^2$ appear in $f_2$, respectively, by Lemma~\ref{lift}.
However neither of the latter monomials appear in $f_2$
 by Lemma~\ref{ilk} and Proposition~\ref{double}. The third monomial
 does not appear in $f_2$ by
Lemma~ \ref{shift}. Finally, $Mx_{m-2}y_{m-2}x_my_m$ appears in $f_2$ if and only if $My_{m-2}^2x_my_m$
 appears in $f_2$ because these are the only monomials in $S$ that
 produce $Mx_{m-2}^2x_my_m$ after applying $\del_1$. However $My_{m-2}^2x_my_m$ appears in
 $f_2$ if and only if $My_{m-2}^2y_m^2$ appears in $f_2$ by Lemma~\ref{lift} and
 the latter monomial does not appear  in $f_2$ by Proposition~\ref{double}.
\end{proof}

Returning to the proof of Theorem~\ref{even-indec},
 first assume that $m\ge 4$.
Then by Lemma~\ref{even-a}, $a_{m-2}$ and $a_{m-1}$ are both even.
Therefore $b_{m-2}$ and $b_{m-1}$ are both odd, contradicting
Lemma~\ref{parity}.

Suppose $m=3$ and $M_1$ appears in $f_1\in S^G$. By Lemma~\ref{even-a}, $a_2$ is even.
Thus $b_2$ is odd and, by Lemma~\ref{parity}, $b_1$ is even.
Therefore $b_1=0$, $a_1=1$ and $M_1=y_1y_2^{a_2}x_3$.
By Lemma~\ref{permute}, $x_1y_2^{a_2}y_3$
 also appears in $f_1$. Thus $y_2^{a_2+1}x_2$ appears in $f_1$ as
 well by Lemma~\ref{shift}. This is contradicts Lemma~\ref{lift}
if $a_2=2$ and Proposition~\ref{double} if $a_2=0$.

\medskip
\subsection{\rm The Proof of Theorem~\ref{zero-even-indec}}\label{zero-even-indec-proof}
\hfil

\medskip
Suppose, by way of contradiction, that $\tr(y_1\cdots y_{\ell'}y_{\ell'+1}^3\cdots y_m^3)$ is
decomposable. Working modulo the $G$-stable ideal
$(x_1,\ldots,x_{m-2},x_{m-1}^2)S$, a straight-forward calculation gives
$$\lt(\tr(y_1\cdots y_{\ell'}y_{\ell'+1}^3\cdots y_m^3))=y_1\cdots  y_{\ell'}y_{\ell'+1}^3\cdots y_{m-1}^3x_{m-1} x_m^2.$$
Thus there are two monomials of positive degree, say $M_1$ and $M_2$, such that
$M_1M_2=y_1\cdots  y_{\ell'}y_{\ell'+1}^3\cdots y_{m-1}^3x_{m-1} x_m^2$,
and both $M_1$ and $M_2$ appear in $G$-invariant polynomials, say $f_1$ and $f_2$.
Without loss of generality, we may assume
$M_1=y_1^{a_1}\cdots y_{m-1}^{a_{m-1}}x_{m-1}x_m^{a_m}$ and $M_2=y_1^{b_1}\cdots y_{m-1}^{b_{m-1}}x_m^{b_m}$.
It follows from Lemma~\ref{ilk} and Proposition~\ref{double}, that $b_m>0$.

\begin{slemma}
If $m>i>1$, then $b_i$ is even and $a_i$ is odd.
\end{slemma}
\begin{proof}

Note that $V^*_{m,0}$ and $(m-1)V_2\oplus 2V_1$ are isomorphic
  $\sigma_2$-modules, where  the two copies of $V_1$ are generated
  by $x_m$ and $y_1$ and each pair $x_{i-1},y_i$ for $2\le i\le m$ generate a copy of $V_2$.
  Therefore we have $S^{\sigma_2}\cong \field[x_1, \dots , x_{m-1}, y_2, \dots , y_m]^{\sigma_2}\otimes
  \field[x_m,y_1]$. Hence the fact that $b_i$ is even follows from Lemma \ref{ilk} (with $\sigma=\sigma_2$).
Since $b_i$ is even and $a_i+b_i$ is odd, $a_i$ is odd.
\end{proof}

 We have $b_m>0$ and $a_m+b_m=2$. Therefore, there are two cases, $a_m=0$ and $a_m=1$.
First assume that $a_m=0$. If $a_{m-1}=3$, then $M_1$ does not appear in $f_1$ by Lemma~\ref{lift}.
On the other if $a_{m-1}=1$, then by Lemma~\ref{lift}, $y_1^{a_1}\cdots y_{m-1}^{a_{m-1}+1}$ appears in $f_1$,
 contradicting  Lemma \ref{ilk} because $a_{m-2}$ is odd.

 Suppose that $a_{m}=1$. Set $M=y_1^{a_1}\cdots y_{m-1}^{a_{m-1}-1}$ so that
$M_1=My_{m-1}x_{m-1}x_m$. Then
 $Mx_{m-2}x_{m-1}x_m$ appears in $\del_2 (M_1)$. The only other
 monomials in $S$ that produce $Mx_{m-2}x_{m-1}x_m$ after applying
 $\del_2$ are $My_{m-2}y_mx_m$ and $Mx_{m-2}y_mx_m$. However by Lemma \ref{lift}
 $My_{m-2}y_mx_m$ appears in $f_1$ if and only if $My_{m-2}y_m^2$
 does but the latter monomial does not appear in $f_1$ by Lemma
 \ref{ilk} and Proposition \ref{double}. Finally, if
 $Mx_{m-2}y_mx_m$ appears in $f_1$, there must be another monomial
 in $f_1$ that produces $Mx_{m-2}x_m^2$ after applying $\del_1$. Since $a_{m-2}$ is
 odd, $Mx_{m-2}y_m^2$ is the only such monomial. However if
 $a_{m-2}=3$, then $Mx_{m-2}y_m^2$ does not appear in $f_1$. If
 $a_{m-2}=1$, then again by Lemma \ref{lift}, $My_{m-2}y_m^2$ also
 appears in $f_1$, contradicting  Proposition~\ref{double}.

\section{The Easy Odd case}
\label{easy}

In this section we consider the odd dimensional representations
$\Omega^{-m}(\field)$. The action of $G$ on
$S:=\field[\Omega^{-m}(\field)]=\field[x_1,\ldots,x_m,y_1,\ldots,y_{m+1}]$
is given by $\sigma_i(x_j)=x_j$, $\sigma_1(y_j)=y_j+x_j$ and
$\sigma_2(y_j)=y_j+x_{j-1}$, using the convention that $x_0=0$ and
$x_{m+1}=0$. As in section~\ref{even-sec}, define
$n_i:=y_i^2+x_iy_i$ and $u_{ij}=x_iy_j+x_jy_i$. Then
$n_i,u_{ij}\in\ring^{\sigma_1}$. A simple calculation gives
$\Delta_2(n_i)=x_{i-1}^2+x_ix_{i-1}$ and
$\Delta_2(u_{ij})=x_ix_{j-1}+x_{i-1}x_j$. For
$i\in\{1,\ldots,m+1\}$ define
$$N_i:=n_i
+\sum_{j=1}^{i-1}\left(u_{i-j,i+j}+u_{i-j,i+j-1}\right),$$
so that $N_1=n_1$ and $N_2=n_2+u_{12}+u_{13}$.
An explicit calculation, exploiting the fact that $\Delta_2(u_{1j})=x_1x_{j-1}$,
gives $\Delta_2(N_i)=0$. Therefore $N_i\in S^G$.
Define $\mathcal H_{-m}:=\{x_1,\ldots,x_m,N_1,\ldots,N_{m+1}\}$.
Since $\lm(N_i)=y_i^2$, $\mathcal H_{-m}$ is a block hsop with top class $y_1\cdots y_{m+1}$
and the image of the transfer is generated by $\tr(\beta)$ for $\beta$ dividing $y_1\cdots y_{m+1}$.

\begin{theorem}\label{easy-odd-thm}
For $m>3$, $\tr(y_1\cdots y_{m+1})$ is indecomposable.
\end{theorem}
See subsection~\ref{easy-odd-proof} for the proof of Theorem~\ref{easy-odd-thm}.
Combining Theorem~\ref{easy-odd-thm} with Theorem~\ref{block-hsop-thm} gives the following.

\begin{corollary}\label{easy-odd-cor} If $m>3$, then the Noether number for $\Omega^{-m}(\field)$ is $m+1$
and the Hilbert ideal is generated by $\mathcal H_{-m}$.
\end{corollary}

 Remarks~\ref{lowVremark} and \ref{easy-three} show that the above formula for the Noether number is valid for $m\geq 1$.

As in section~\ref{even-sec}, define $t_j:=u_{12}x_{j-1}+u_{1j}x_1$.
%and $\widetilde{w}:=(x_1+x_2)u_{12}+x_1n_2$.
%Observe that
%$t_j\in\field[x_1,x_2,x_j,y_1,y_2,y_j]$ has degree one as a polynomial
%in $y_j$ and that the coefficient of $y_j$ in $t_j$ is $x_1^2$.

\begin{theorem}\label{FF-Vn-Theorem} For $m>2$,
$$\field[\Omega^{-m}(\field)]^G[x_1^{-1}]=\field[x_1,\ldots,x_m,N_1,N_2,t_{3},\ldots,t_{m+1}][x_1^{-1}].$$
\end{theorem}
\begin{proof}
We construct the field of fractions for an upper-triangular action as in \cite{CampChuai}
or \cite{Kang}. The restriction of the action of $G$ to the span of $\{x_1,x_2,y_1,y_2\}$
is $V_{2,0}^*$.
Therefore, using  Remark~\ref{zero-two_remark},
$\field[x_1,x_2,y_1,y_2]^G[x_1^{-1}]=\field[x_1,x_2,n_1,\widetilde{w}]^G[x_1^{-1}]$.
%Thus at this stage, the calculation follows from Remark~\ref{zero-two_remark}.
Since $t_j\in\field[x_1,\ldots,x_m,y_1,\ldots,y_j]^G$
has degree one as a polynomial in $y_j$ with coefficient $x_1^2$,
we have $\field[\Omega^{-m}(\field)]^G[x_1^{-1}]=\field[x_1,\ldots,x_m,n_1,\widetilde{w},t_{3},\ldots,t_{m+1}][x_1^{-1}].$
The result then follows from the fact that $\widetilde{w}=x_1N_2+t_3$ and $N_1=n_1$.
\end{proof}

\begin{remark}\label{lowVremark} It is easy to see that
$\field[\Omega^{-1}(\field)]^G=\field[x_1,n_1,y_2^2+x_1y_2]$.
A Magma calculation shows that $\field[\Omega^{-2}(\field)]^G$ is the hypersurface with generators
$x_1$, $x_2$, $N_1$, $N_2$, $N_3$, $t_3$ and relation
$t_3^2+x_2^4N_1+x_1x_2(x_1+x_2)t_3+x_1^2x_2^2N_2=x_1^4N_3$.
Therefore, the Noether number for this representation is $m+1=3$.
%Furthermore, this generating set is a SAGBI bases using grevlex with
%$y_2>y_1>x_3>x_2>x_1$.
%Note that $\widetilde{w}=x_1N_2+t_3$.
% Therefore, we could replace $\widetilde{w}$ with
%$N_2$ in Theorem~\ref{FF-Vn-Theorem}.
\end{remark}

\begin{remark}
It follows from Theorem~\ref{FF-Vn-Theorem} % and Remark~\ref{lowVremark}
that applying the SAGBI/Divide-by-x algorithm of \cite{CSW} with $x=x_1$ to
$$\{x_1,\ldots,x_m,N_1,N_2,\ldots,N_{m+1},t_3,\ldots,t_{m+1}\}$$
produces a SAGBI basis for $\field[\Omega^{-m}(\field)]^G$.
\end{remark}

\begin{remark}\label{easy-three}
A Magma calculation shows that $\field[\Omega^{-3}(\field)]^G$ is generated by
$$\left\{x_1,x_2,x_3,n_1,N_2,N_3,n_4,t_3,t_4,u_{233},u_{133},\tr(y_1y_2y_3y_4)\right\}$$
where $u_{133}:=x_3u_{13}+x_1u_{24}$ and $u_{233}:=x_3u_{23}+x_2u_{24}+x_3u_{14}$.
Furthermore, this set is a SAGBI basis and running a SAGBI test with $\tr(y_1y_2y_3y_4)$ omitted, shows
that $\tr(y_1y_2y_3y_4)$ is indecomposable.
Therefore the Noether number for this representation is $m+1=4$ and the Hilbert ideal is generated by
the block hsop $x_1,x_2,x_3,n_1,N_2,N_3,n_4$.
From \cite{El-Fl}, we know ${\rm depth}(\field[\Omega^{-3}(\field)]^G)=6$.
The relation $x_2t_4+x_3t_3+x_1u_{133}=0$
%$x_3v_2+(x_2^2+x_1x_3)n_{13}+ x_1u_{1233}=0$
shows that the
partial hsop $\{x_1,x_2,x_3\}$ is not a regular sequence, giving an alternate proof of the fact
that the ring is not Cohen-Macaulay.
\end{remark}
\medskip

\begin{proposition}
For $\ring=\field[\Omega^{-m}]$,
$(\Delta_2\ring)^{\sigma_1}=(\Delta_1S)^{\sigma_2}
=\left(\left(x_1,\ldots,x_m\right)\ring\right)^G=\sqrt{\tr\ring }$ and
$$\sqrt{\tr \ring}/\tr \ring\cong H^1(\langle\sigma_2\rangle,\Delta_1\ring)=H^1(\langle\sigma_1\rangle,\Delta_2\ring).$$
Furthermore $S^G/\sqrt{\tr S}\cong \field[N_1,\ldots,N_m]$.
\end{proposition}
\begin{proof}
The proof is analogous to the proof of Proposition~\ref{even-coh-prop}.
(Note that $\lt(N_i)=y_i^2$ and so an analogue of Proposition~\ref{double} is unnecessary.)
\end{proof}

\subsection{\rm Proof of Theorem~\ref{easy-odd-thm}}\label{easy-odd-proof}
\hfil

\medskip
Suppose, by way of contradiction, that $\tr(y_1\cdots y_{m+1})$ is
decomposable. Working modulo the $G$-stable ideal
$(x_1,\ldots,x_{m-1})S$, it easy to see that
$$\lt(\tr(y_1\cdots y_{m+1}))=y_1\cdots y_{m-1}x_m^2.$$
Thus there are two monomials, say $M_1$ and $M_2$, such that
$M_1M_2=y_1\cdots y_{m_-1}x_m^2$, $\deg(M_2)\leq\deg(M_1)<m+1$ and
both $M_1$ and $M_2$ appear in $G$-invariant polynomials. Since a
$G$-invariant is also a $\sigma_1$-invariant, it follows from
Lemma~\ref{ilk} that both $M_1$ and $M_2$ are divisible by $x_m$.
Since $m+1\geq 5$, we have $\deg(M_1)\geq 3$. The required
contradiction is then a consequence of the following lemma.

\begin{slemma}
Let $M=(\prod_{j\in J} y_j)x_k$  for some $k\le m$ and  set
$J\subseteq \{1, \dots , k-1\}$ with $|J|>1$.    Then $M$ does not
appear  with a non-zero coefficient in a $G$-invariant polynomial.
\end{slemma}
\begin{proof}
Let $d$ denote the maximum integer in $J$. We proceed by
induction on $k-d$. Assume on the contrary that $M$ appears in a
$G$-invariant polynomial $f$.  Set $M'=\prod_{j\in J, \; j\neq d}
y_j$. Then we have $M=M'y_dx_k$. From  Lemma \ref{permute} we get
that  $M'x_dy_k$ also appears in $f$.
Furthermore, since $M'x_dx_{k-1}$
appears in $\Delta_2(M'x_dy_k)$, there must be
another monomial in $f$ that produces $M'x_dx_{k-1}$ after
applying $\Delta_2$.
If $k-d=1$, then the only other monomial that
produces $M'x_dx_{k-1}=M'x_d^2$ after applying $\Delta_2$ is
$M'y_k^2$. However, this monomial can not appear in $f$ by Lemma \ref{ilk}. This
establishes the basis case for the induction. If $k-d>1$, the only
monomials (other than $M'x_dy_k$) that produce
$M'x_dx_{k-1}$ after applying $\Delta_2$ are $M'y_{d+1}y_k$ and $M'y_{d+1}x_{k-1}$.
Again by Lemma~\ref{ilk}, $M'y_{d+1}y_k$ can not appear in $f$.
Moreover, if $d+1<k-1$, then $M'y_{d+1}x_{k-1}$ does not appear in $f$ by
induction.
% Because for this monomial the largest index of a
%$y$-variable is larger than that of $M$ and the index of the
%$x$-variable is smaller than that of $M$.
On the other hand, if $d+1=k-1$, then $M'y_{d+1}x_{k-1}$ does not appear in $f$  by
Lemma~\ref{ilk}.
\end{proof}

\section{The Hard Odd Case}
\label{hard}

In this section we consider the odd dimensional representations
$\Omega^{m}(\field)$. The action of $G$ on
$S:=\field[\Omega^{m}(\field)]=\field[x_1,\ldots,x_{m+1},y_1,\ldots,y_{m}]$
is given by $\sigma_i(x_j)=x_j$, $\sigma_1(y_j)=y_j+x_j$ and
$\sigma_2(y_j)=y_j+x_{j+1}$. Define $$\mathcal
H_{m}:=\{x_1,\ldots,x_{m+1},\norm{y_1},\ldots,\norm{y_m}\}.$$

Since $\lm(\norm{y_i})=y_i^4$, $\mathcal H_{m}$ is a block hsop with top class $(y_1\cdots y_{m})^3$
and the image of the transfer is generated by $\tr(\beta)$ for $\beta$ dividing $(y_1\cdots y_{m})^3$.

\begin{theorem}\label{hard-odd-thm}
For $m>2$, $\tr(y_1^3\cdots y_{m}^3)$ is indecomposable.
\end{theorem}
See subsection~\ref{hard-odd-proof} for the proof of Theorem~\ref{hard-odd-thm}.
Combining Theorem~\ref{hard-odd-thm} with Theorem~\ref{block-hsop-thm} gives the following.

\begin{corollary}\label{hard-odd-cor} If $m>2$, then the Noether number for $\Omega^{m}(\field)$ is $3m$
and the Hilbert ideal is generated by $\mathcal H_{m}$.
\end{corollary}

From Remark~\ref{hard-two}, the Noether number for $\Omega^{2}(\field)$ is $6$.

For $j>1$, define $v_j:=u_{1j}(x_2^2+x_1x_2)+n_1(x_jx_2+x_1x_{j+1})$.

\begin{theorem} For $m>1$,
$$\field[\Omega^m]^G[(x_1x_2(x_1+x_2))^{-1}]
=\field[x_1,\ldots,x_{m+1},N_G(y_1),v_2,\ldots,v_m] [(x_1x_2(x_1+x_2))^{-1}].$$
\end{theorem}
\begin{proof}
We use \cite[Theorem~2.4]{CampChuai}. $\field[x_1,\ldots,x_m,y_1]^G$
is the polynomial algebra generated by
$\{x_1,\ldots,x_m,N_G(y_1)\}$. The invariant $v_j\in\field[x_1,x_2,x_j,x_{j+1},y_1,y_j]$
has degree one as a polynomial in $y_j$ and the coefficient of $y_j$ is
$x_1x_2(x_1+x_2)$.
\end{proof}

It is easy to see that $\field[\Omega^1(\field)]^G=\field[x_1,x_2,N_G(y_1)]$
and, therefore, the Noether number is $4$.

\begin{remark}\label{hard-two}
A Magma calculation shows that $\field[\Omega^2(\field)]^G$ is generated by
$$\mathcal B_2:=\{x_1, x_2, x_3, \norm{y_1}, \norm{y_2},  v_2, n_{13}, u_{1233}, \tr(y_1^3y_2^3)\}$$
where $n_{13}=x_3n_1+x_3u_{12}+x_1n_2$ and
$u_{1233}=(x_3^2+x_2x_3)u_{12}+(x_2^2+x_1x_3)n_2$.
Therefore the Hilbert ideal for $\Omega^2(\field)$ is generated by $x_1,x_2,x_3,\norm{y_1},\norm{y_2}$.
In fact, $\mathcal B_2$ a SAGBI bases using grevlex with $y_2>y_1>x_3>x_2>x_1$.
Although a SAGBI basis need not be a minimal generating set, running a SAGBI basis test on
$\mathcal B_2\setminus\{\tr(y_2^3y_3^3)\}$ shows that $\tr(y_2^3y_3^3)$ is indecomposable and
hence the Noether number is $6$.
From \cite{El-Fl}, we know ${\rm depth}(\field[\Omega^2(\field)]^G)=4$.
The relation $x_3v_2+(x_2^2+x_1x_3)n_{13}+ x_1u_{1233}=0$ shows that the
partial hsop $\{x_1,x_2,x_3\}$ is not a regular sequence, giving an alternate proof of the fact
that the ring is not Cohen-Macaulay.
\end{remark}

\begin{remark}
We have been unable to find ``polynomial generators'' for the ring $\field[\Omega^m(\field)]^G[x_1^{-1}]$.
We note that $x_1$ is not in the radical of the image of the transfer for these representations
but that  $x_1x_2(x_1+x_2)$ is.
Furthermore, $x_1$ is in the radical of the image of the transfer for $\Omega^{-m}(\field)$ and $V_{m,\lambda}$.
Hence $\field[\Omega^{-m}]^G[x_1^{-1}]$ and $\field[V_{m,\lambda}]^G[x_1^{-1}]$ are ``trace-surjective''
in the sense of \cite{Fl-W}.
\end{remark}

\begin{proposition}
For $\ring=\field[\Omega^{m}(\field)]$ and $m\geq 3$,
$$\sqrt{\tr\ring}=((x_2x_{m+1}+x_2x_1,x_1x_{m+1}+x_1x_2,x_2^2+x_2x_1,x_3+x_2,\ldots,x_m+x_2)S)^G.$$
\end{proposition}
\begin{proof}
 Direct calculation gives
$\Delta_1 (\Omega^m(F)^*)={\rm Span}_{\field}\{x_1,\ldots,x_m\}$,
$\Delta_2 (\Omega^m(F)^*)={\rm Span}_{\field}\{x_2,\ldots,x_{m+1}\}$, and
$(\sigma_1\sigma_2+1) (\Omega^m(F)^*)={\rm Span}_{\field}\{x_1+x_2,\ldots,x_m+x_{m+1}\}$.
Using \cite[Theorem~2.4]{modular-transfer} and computing intersections of ideals gives
\begin{eqnarray*}
\sqrt{\tr\ring}&=&\bigcap_{g\in G,\,|g|=2}(((g-1)\Omega^m(\field)^*)S)^G\cr
&=&((x_2x_{m+1}+x_2x_1,x_1x_{m+1}+x_2x_1,x_2^2+x_2x_1,x_3+x_2,\ldots,x_m+x_2)S)^G.
\end{eqnarray*}
\end{proof}

\begin{remark} The above shows that for $m\geq 3$, we have $x_2+x_3\in\sqrt{\tr\ring}$.
In fact, for $$\alpha:=(x_1+x_2+x_3)y_2y_3+(x_1+x_2+x_3+x_4)y_1y_3+(x_2+x_3+x_4)y_1y_2+y_1^2y_3+y_1y_3^2,$$
$\tr(\alpha)=(x_2+x_3)^3$.
Define $x:=x_2+x_3$ and use the variables $x<x_1<x_3<x_4<\cdots<x_{m+1}<y_1<\cdots<y_m$
with the grevlex order. Define $\rho:\field[\Omega^m(\field)][x^{-1}]\to\field[\Omega^m(\field)]^G[x^{-1}]$
by $\rho(f)=x^{-3}\tr(f\alpha)$. Then $\rho$ restricts to the identity on $\field[\Omega^m(\field)]^G$
and  $\field[\Omega^m(\field)]^G[x^{-1}]$ is ``trace-surjective''.
Define $$\mathcal{B}_m:=\mathcal{H}_m\cup\{\tr(\beta)\mid \beta\, {\rm divides}\, (y_1\cdots y_m)^3\}.$$
Since $\{\beta\mid \beta\, {\rm divides}\, (y_1\cdots y_m)^3\}$ generates $\field[\Omega^m(\field)][x^{-1}]$
as a module over the ring $\field[\mathcal H_m][x^{-1}]$ and $\rho$ is surjective,
we see that  $\mathcal B_m\cup\{x^{-1}\}$ generates $\field[\Omega^m(\field)]^G[x^{-1}]$. Thus,
since $\mathcal H_m$ is an hsop, applying the SAGBI/Divide-by-$x$ algorithm to $\mathcal B_m$
produces a generating set, in fact a SAGBI basis, for $\field[\Omega^m(\field)]^G$.
\end{remark}

\medskip

\subsection{\rm Proof of Theorem~\ref{hard-odd-thm}} \label{hard-odd-proof}
\hfil

\medskip
Suppose, by way of contradiction,
that $\tr(y_1^3\cdots y_{m}^3)$ is decomposable.
%As in the previous case, we obtain a contradiction by showing that  its
%leading term can not be written as a product two positive degree
%monomials that appear in $G$-invariant polynomials.
Working modulo the $G$-stable ideal $(x_1, \dots ,x_{m-1})S$,
it is not difficult to see that
$$\lt (\tr(y_1^3\cdots y_{m}^3))=y_1^3\cdots y_{m-1}^3x_{m}x_{m+1}^2.$$
Write $y_1^3\cdots y_{m-1}^3x_{m}x_{m+1}^2=M_1M_2$, where $M_1$
and $M_2$ are monomials of positive degree which appear in
$G$-invariant polynomials. We use the following results to
eliminate possible factorisations.

\begin{slemma}
\label{delta2} Suppose  $1\le i \le m$, $2 \le k \le m+1$, $k\neq
i+1$ and  $M$ is a monomial in $y_1,\ldots,y_m$.
Further suppose that the degree of $y_i$ in $M$
is even and $y_ix_kM$  appears in a $G$-invariant
polynomial $f$. Then the degree of $y_{k-1}$ in $M$ is even and
$x_{i+1}y_{k-1}M$  appears in $f$.
\end{slemma}
\begin{proof}
Since the degree of $y_i$ in $M$ is even, $x_{i+1}x_kM$ appears in $\Delta_2 (y_ix_kM)$. Since
$\Delta_2 (f)=0$, $f$ must contain another monomial that produces
$x_{i+1}x_kM$ after applying $\Delta_2$. If the degree of
$y_{k-1}$  in $M$ is odd, then there is no such monomial.
Thus the degree of $y_{k-1}$ in $M$ is even and
applying $\Delta_2$ to  either $y_iy_{k-1}M$ or
 $x_{i+1}y_{k-1}M$ produces $x_{i+1}x_kM$. However, by Lemma~\ref{ilk}, $y_iy_{k-1}M$
does not appear in $f$. Thus  $x_{i+1}y_{k-1}M$ appears in $f$.
\end{proof}

\begin{sproposition}
\label{artikiki} Suppose $M=y_1^{e_1}\cdots y_m^{e_m}$. If $k$ is
a positive integer and  $Mx_1^k$ or $Mx_{m+1}^k$ appears in a
$G$-invariant polynomial, than $e_j$ is even for $1\le j\le m$.
\end{sproposition}

\begin{proof}
Note that $S^{\sigma_1}\cong\field[x_i, y_i\mid i\le m]^{\sigma_1}\otimes \field[x_{m+1}]$
and  $S^{\sigma_2}\cong\field[x_{i+1}, y_i\mid i\le m]^{\sigma_2}\otimes \field[x_{1}]$.
If $Mx_{m+1}^k$ appears in a $G$-invariant polynomial, then
$M$ appears in a $\sigma_1$-invariant polynomial, and
the result follows from applying Lemma~\ref{ilk} with $\sigma=\sigma_1$.
If $Mx_{1}^k$ appears in a $G$-invariant polynomial, then
$M$ appears in a $\sigma_2$-invariant polynomial, and
the result follows from applying Lemma~\ref{ilk} with $\sigma=\sigma_2$.
\end{proof}

\begin{sproposition}
\label{ciftiki} Suppose $M=\prod_{j\in J}y_j^2$ for a non-empty index set
$J\subseteq \{1, \dots , m\}$. Then $M$ does not
appear in a $G$-invariant polynomial.
\end{sproposition}
\begin{proof}
Suppose, by way of contradiction, that $M$ appears in a $G$-invariant polynomial $f$.
Let $\ell$ denote the largest integer in $J$ and set $M'=M/y_{\ell}^2$.
%\prod_{j\in J, \; j\neq l}y_j^2$.
Note that $M'x_{\ell+1}^2$ appears in $\Delta_2(M)$ and since $\Delta_2 (f)=0$, there must be another monomial in
$f$ that produces $M'x_{\ell+1}^2$ after applying $\Delta_2$.
The only other monomial in $S$ with this property is $M'y_{\ell}x_{\ell+1}$.
Therefore, this monomial also appears in $f$.
If $\ell=m$, then the degree of $y_m$ in
$M'y_{\ell}x_{\ell+1}=M'y_mx_{m+1}$ is odd, and we have a
contradiction by Proposition~\ref{artikiki}.
Otherwise, using Lemma~\ref{permute},
$M'x_{\ell}y_{\ell+1}$ appears in $f$. If $\ell=1$, this also gives a
contradiction using  Proposition~\ref{artikiki}.
Otherwise, we apply Lemma~\ref{delta2} and conclude that $M'y_{\ell-1}x_{\ell+2}$ appears in $f$.
This gives a contradiction if $\ell+1=m$. Continuing in
this fashion, the process terminates with either $M'y_{2\ell-m}x_{m+1}$ or
$M'y_{2\ell}x_{1}$ appearing in $f$, again contradicting Proposition~\ref{artikiki}.
\end{proof}

Returning  to the proof of Theorem~\ref{hard-odd-thm},
first suppose that $M_1$ is a  factor
of $y_1^3\cdots y_{m-1}^3$. Since $M_1$ appears in a
$\sigma_1$-invariant, we have from Lemma~\ref{ilk} that
the degree of each $y_i$ in $M_1$ is even.
However, since these degrees are at most two, we get a
contradiction using Proposition~\ref{ciftiki}.
Similarly, $M_2$  is a
not factor of $y_1^3\cdots y_{m-1}^3$.
Therefore we may assume $x_m$
divides $M_1$ and $x_{m+1}$ divides $M_2$.
By Proposition~\ref{artikiki},
 the degrees of the variables $y_1, \ldots, y_{m-1}$  in $M_2$ are even. Hence
 the degrees of these variables in $M_1$ are odd. Therefore we have either
 $M_1=y_1^{a_1}\cdots y_{m-1}^{a_{m-1}}x_m$ or $M_1=y_1^{a_1}\cdots y_{m-1}^{a_{m-1}}x_mx_{m+1}$,
  where $a_1, \dots ,a_{m-1}$ are odd. Let $f$ denote the $G$-invariant polynomial in which $M_1$
  appears. Suppose that  $M_1=y_1^{a_1}\cdots y_{m-1}^{a_{m-1}}x_m$. Since
  $y_1^{a_1}\cdots y_{m-1}^{a_{m-1}-1}x_m^2$ appears in $\Delta_2 (M_1)$ and $\Delta_2 (f)=0$,
  there must be another monomial in $f$ that produces $y_1^{a_1}\cdots y_{m-1}^{a_{m-1}-1}x_m^2$ after
   applying $\del_2$. However, $y_1^{a_1}\cdots y_{m-1}^{a_{m-1}+1}$ is the only other monomial in $S$ with
    this property. Since $f$ is also $\sigma_1$-invariant and $a_1$ is odd we get a contradiction by Lemma~\ref{ilk}.
Therefore, we may assume that $M_1=y_1^{a_1}\cdots y_{m-1}^{a_{m-1}}x_mx_{m+1}$.
Then $y_1^{a_1}\cdots y_{m-1}^{a_{m-1}-1}x_m^2x_{m+1}$ appears in
$\del_2 (M_1)$. Since $\Delta_2 (f)=0$, there must be another
monomial in $f$ that produces $y_1^{a_1}\cdots
y_{m-1}^{a_{m-1}-1}x_m^2x_{m+1}$ after applying $\del_2$. The
monomials in $S$ with this property are  $y_1^{a_1}\cdots
y_{m-1}^{a_{m-1}+1}y_{m}$,
 $y_1^{a_1}\cdots y_{m-1}^{a_{m-1}}x_my_{m}$, $y_1^{a_1}\cdots
y_{m-1}^{a_{m-1}+1}x_{m+1}$, $y_1^{a_1}\cdots
y_{m-1}^{a_{m-1}-1}x_m^2y_{m}$.
 The first two monomials do not appear in $f$  by Lemma \ref{ilk}
 because the degree of $y_1$ is odd. For the same reason the third monomial does not
 appear in $f$ by Proposition \ref{artikiki}. Finally, if $y_1^{a_1}\cdots
y_{m-1}^{a_{m-1}-1}x_m^2y_{m}$ appears in $f$, then there must be
another monomial in $f$ that produces $y_1^{a_1-1}x_1\cdots
y_{m-1}^{a_{m-1}-1}x_m^2y_{m}$ after applying $\del_1$. However,
$y_1^{a_1}\cdots y_{m-1}^{a_{m-1}-1}y_{m}^3$ and
$y_1^{a_1-1}x_1\cdots y_{m-1}^{a_{m-1}-1}y_{m}^3$ are the only
monomials in $S$ with this property. Since neither of these
monomials can appear in $f$, by Lemma~\ref{ilk} and
Proposition~\ref{artikiki} respectively, we have ruled out all
possible factorisations, proving Theorem~\ref{hard-odd-thm}.

\begin{ack}
 We thank  the referee for  comments that lead to the improvement of the exposition of the paper.
\end{ack}

%%%%%%%%%%%%%%%%%%%%%%%%%%%%%%%%%%%%%%%%%%%%%%%%%%%%%%%%%%%%%%%%%%%%%%%%%
\ifx\undefined\bysame
\newcommand{\bysame}{\leavevmode\hbox to3em{\hrulefill}\,}
\fi

\end{document}